\documentclass[11pt]{amsart}
\usepackage[utf8x]{inputenc}
\usepackage{amsfonts}
\usepackage{amssymb}
\usepackage{amsmath}
\usepackage{delarray}
\usepackage{graphicx}
\usepackage[all]{xy}
\newtheorem{thm}{Theorem}[section]

\newtheorem{prop}[thm]{Proposition}

\newtheorem{defn}[thm]{Definition}
\theoremstyle{remark}
\newtheorem{remark}[thm]{Remark}
\numberwithin{equation}{section}

\newcommand{\cK}{\mathcal K}
\newcommand{\cN}{\mathcal N}
\newcommand{\cF}{\mathcal F}

\newcommand{\cC}{\mathcal C}
\newcommand{\cB}{\mathcal B}

\newcommand{\cL}{\mathcal L}
\newcommand{\cO}{\mathcal O}

\newcommand{\cU}{\mathcal U}

\newcommand{\cS}{\mathcal S}

\newcommand{\bbR}{\mathbb R}
\newcommand{\bbT}{\mathbb T}
\newcommand{\bbC}{\mathbb C}

\newcommand{\bbN}{\mathbb N}

\begin{document}

%\newpage
%opening
\title[Geometry of integrable systems on 2D surfaces]{Geometry of integrable dynamical systems on 2-dimensional surfaces}

\author{Nguyen Tien Zung and Nguyen Van Minh}
\address{Institut de Mathématiques de Toulouse, UMR5219, Université Toulouse 3}
\email{tienzung.nguyen@math.univ-toulouse.fr, minh@math.univ-toulouse.fr}

\date{1st version, 06/April/2012}

\begin{abstract} 
This paper is devoted to the problem of classification, up to smooth isomorphisms or up to orbital equivalence, 
of smooth integrable vector fields on  2-dimensional surfaces,  under some nondegeneracy conditions. The main continuous
invariants involved in this classification are the left equivalence classes of period
or monodromy functions, and the cohomology classes of period
cocycles, which can be expressed in terms of Puiseux series. 
We also study the problem of Hamiltonianization of these integrable vector fields by a compatible symplectic
or Poisson structure.
\end{abstract}

\maketitle

{\small \tableofcontents }

\section{Introduction and preliminaries}

The aim of this paper, which is part of our program of systematic study of the geometry and topology 
of integrable non-Hamiltonian  dynamical systems \cite{AyoulZung_Galois2010,Zung-Convergence2002,Zung-Toric2003,
Zung-Nondegenerate2012,Zung-SmoothLinearzation2012,ZungMinh-Action2012}, 
is to describe the local and global invariants and classification of smooth vector fields on 2-dimensional
surfaces, which admit a non-trivial first integral. 
%In the language of integrable dynamical systems 
%(see, e.g. \cite{Zung-Convergence2002}, \cite{Zung-Nondegenerate2012}) 
Such a vector field, together with a first integral, is called an {\bf integrable dynamical system of type (1,1)} 
(i.e. 1 vector field and 1 function). 
A special case of systems of type (1,1) is Hamiltonian systems on symplectic surfaces, 
where the Hamiltonian function itself is a first integral of the Hamiltonian vector field.
Invariants of Hamiltonian systems on surfaces have been studied by many people, in particular Fomenko \cite{Fomenko-Hamiltonian1987}
who introduced the notion of ``atoms'' and ``molecules'' for semi-local and global topological classification of these systems and systems
with $1\frac{1}{2}$ degrees of freedom, and Dufour - Molino -Toulet \cite{DufourMolino-Class2dim1994} 
who gave a symplectic classification in terms of Taylor series of regularized  action functions. We will extend the known ideas and results 
in the Hamiltonian case to the general non-Hamiltonian case.

We will denote an integrable system of type (1,1) by a couple $(X,F)$ or $(X,\cF)$, where $X$ is a vector field on a 2-dimensional surface 
$\Sigma$ such that $X \neq 0$ almost everywhere, $F$ is a first integral of $X$ (i.e. $X(F)=0$) such that $dF \neq 0$ almost everywhere,
and $\cF$ is the ring of all first integrals of $X$. The functional dimension of $\cF$ is 1, i.e.  $dF\wedge dG = 0$  for any $F, G \in \cF.$ 
The main object of our study is $X$ and not $F$ : $X$ is fixed while $F$ can be replaced by any other appropriate first integral.
A point $p \in \Sigma$ is called {\bf singular} if it is singular with respect to $X$, i.e. $X(p) = 0$.
The couple  $(X,\cF)$ gives rise to a singular 1-dimensional {\bf associated singular fibration} on the ambient surface $\Sigma$:
each fiber is a maximal connected subset of  $\Sigma$ on which every first integral $F \in \cF$ 
is constant.  Each fiber of this fibration is called a {\bf level set} of $(X,\cF)$. 
Notice that the level sets are invariant under the flow of $X$. 
A level set $\cN$ is called {\bf regular} if there is a first integral $F \in \cF$ such that $F$ is regular on $\cN$, i.e. $dF \neq 0$ everywhere on $\cN$.
Remark that a regular level set may contain singular points of $X$. 

We will study integrable systems $(X,\cF)$ locally, i.e. near a point, 
semi-locally, i.e. in the neighborhood of a level set, and globally, i.e. on the whole surface. 
We will describe local and semi-local invariants of these systems, 
which allow us to classify them up to smooth isomorphisms or smooth orbital equivalence, in the following sense:

\begin{defn}
Two smooth integrable systems $(X_1,F_1)$ and $(X_2,F_2)$ on two surfaces $\Sigma_1$ and $\Sigma_2$ 
respectively are called smoothly {\bf orbitally equivalent} if there is a smooth diffeomorphism $\Phi : \Sigma_1 \to \Sigma_2$  
such that $\Phi_*X_1 \wedge X_2 = 0$ and the set of singular points of $X_2$ (where $X_2=0$) 
coincides with the set of singular points of $\Phi_*X_1$.
They are called smoothly {\bf isomorphic} if $\Phi_*X_1 = X_2$. 
\end{defn}

In the above definition, one may replace the smooth ($C^\infty$) category by some other category, e.g. $C^1, C^k \ (1 \leq k < \infty)$,
$C^\omega$ (analytic). We will work mainly in the smooth and the real analytic categories. In the literature, there are also vector fields
which are integrable in a weaker sense: their first integrals are only smooth outside a small set, see e.g. \cite{Llibre-IntegrablePlanar2011}. Here
we require the first integral to be smooth everywhere. 
Moreover, we will restrict our attention to smooth integrable  systems $(X,\cF)$ which are \emph{weakly nondegenerate} in the sense of Definition 
\ref{defn:weakly_nondegenerate} below. It means that each singular point of $(X,\cF)$ is either \emph{nondegenerate}
(in the sense of \cite{Zung-Nondegenerate2012,Zung-SmoothLinearzation2012}), 
or generic nilpotent (see below). Nondegenerate singular points can be classified into 3 types depending on the eigenvalues of $X$ at them,
so in total we allow the following 4 types of singular points in our systems:

\begin{enumerate}
\item Type I. Elliptic: two purely imaginary eigenvalues.

\item Type II. Hyperbolic with  eigenvalue 0:  two different real eigenvalues, one of which is 0.

\item Type III. Hyperbolic without  eigenvalue 0: two non-zero real eigenvalues.

\item Type IV. Generic nilpotent: the linear part of $X$ is nilpotent non-zero, the quadratic part of $X$ is generic
 and there is a local first integral $F$  with $dF \neq 0$ at the singular point. (See Definition \ref{defn:Generic_Nilpotent}).
\end{enumerate}

These singularities  admit nice local normalizations which linearize the associated singular fibration.
For the nondegenerate singularities (Types I, II, III) this fact was shown in \cite{Zung-Nondegenerate2012,Zung-SmoothLinearzation2012},
and for nilpotent  singularities (Type IV) it follows from the definition and Takens-Gong normal form \cite{Takens-NF1974,Gong-NF1995}. 
Moreover, these 4 types of singularities are
locally structurally stable, i.e.  any $C^1$-close integrable system will have the same types of singular points.

\begin{defn} \label{defn:weakly_nondegenerate}
 A smooth integrable system $(X,\cF)$ on a surface $\Sigma$ is called {\bf weakly nondegenerate} if every of its 
singular points is of one of the above 4 types I-IV.
\end{defn}

\begin{figure}[htb] 
\begin{center}
\includegraphics[width=120mm]{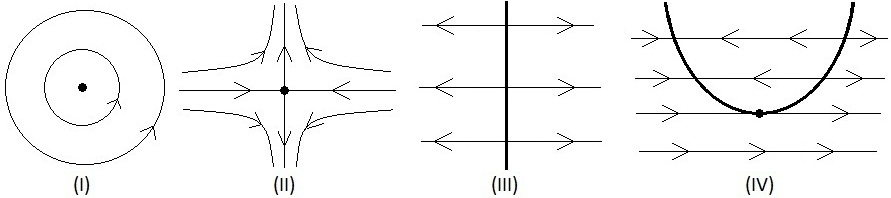}
\caption{Four types of allowed singularities.}
\label{fig:4types}
\end{center}
\end{figure}

We will usually denote by $\cB$ the  base space of the associated singular fibration of an intgerable system $(X,\cF)$: 
each point of $\cB$ corresponds to a level set of $(X,\cF)$. Then we have a natural projection
\begin{equation}
proj : \Sigma \to \cB .
\end{equation}
The topology and the differential structure of $\cB$ is induced from $\Sigma$ via the projection map. 
The induced differential structure was probably first studied by Reeb and Haefliger 
\cite{Reeb-Variete1957} for a similar situation. 
Each first integral $F \in \cF$ descends to a smooth function on $\cB$. 
Imitating Dufour - Molino - Toulet \cite{DufourMolino-Class2dim1994}, we will view
 $\cB$ as a graph and call it the {\bf Reeb graph} of $(X,\cF)$: each vertex of the Reeb graph corresponds to
a level set which contains at least one singular point of Type I, III or IV. 
Of course, the Reeb graph is an important orbital invariant of the system, 
and together with some other discrete invariants
it gives an orbital classification of the systems $(X,\cF)$ similar to a topological clasification 
obtained by Fomenko and his collaborators 
(see, e.g., \cite{BolsinovFomenko-IntegrableBook,Fomenko-Hamiltonian1987, Oshemkov-Morse1995}) 
for  integrable  Hamiltonian systems with 2 degrees of freedom. 

The organization of this paper is as follows: 
In Section 2 we will describe the local invariants of nondegenerate singular points, 
using smooth normal forms given by the geometric linearization. We obtain the local
classification of these singularities in terms of (left equivalence classes of) the local {\bf period functions}
or the {\bf frequency functions}. In Section 3 we give a semi-local classification of nondegenerate level sets
in terms of the {\bf monodromy fuctions} or cohomology classes of the {\bf  period cocycles}. In Section 4 we study generic nilpotent singularities,
where again we obtain a classification in terms of regularized monodromy functions. In Section 5, we put local and
semi-local invariants together on the Reeb graph to get a global classification of weakly-nondegenerate
integrable systems of type (1,1). Finally, in Section 6, we give necessary and sufficient conditions for an integrable
vector field in dimension 2 to be Hamiltonian with respect to some symplectic or Poisson structure.

As will be shown in the paper, our continuous semi-local invariants
are similar in nature and constructed in a similar way to the symplectic and flow invariants studied by
Dufour - Molino - Toulet \cite{DufourMolino-Class2dim1994}, Bolsinov \cite{Bolsinov-Trajectory1995} 
and Kruglikov \cite{Kruglikov-HamiltonianSurfaces1997,Kruglikov-VFSurfaces1999}
for Hamiltonian or isochore systems on 2-dimensional surfaces, and by Bolsinov, Vũ Ngọc San,  and Dullin 
\cite{Bolsinov-Trajectory1995, BolsinovSan-Symplectic2012,DullinSan_Symplectic2007,San_Focus2003}  
for higher-dimensional cases. Nevertheless, our 
invariants are different from and complementary to the invariants in 
\cite{Bolsinov-Trajectory1995, BolsinovSan-Symplectic2012, DufourMolino-Class2dim1994, 
DullinSan_Symplectic2007, Kruglikov-HamiltonianSurfaces1997,Kruglikov-VFSurfaces1999, San_Focus2003},
and instead of being expressed in terms of Taylor series as in the cited papers, they are expressed in terms of Puiseux series, 
due to the non-Hamiltonian nature of the studied systems.

\section{Local structure of nondegenerate singularities}
\subsection{Elliptic singularities (Type I)}

A singular point $p$ (i.e., a point where $X$ vanishes)
of an integrable system $(X,\cF)$ is called a nondegenerate {\bf elliptic} singular point, 
or also a singular point of {\bf Type I} in this paper, if the eigenvalues of $X$ at $p$ are pure imaginary non-zero (i.e.
the linear part of $X$ at $p$ generates a linear circle action on a plane), and there is a first integral $F \in \cF$
which is non-flat at $p$.

Suppose that $(X,\cF)$ is a smooth integrable system with a singular point $p$ of Type I. 
According to  the local geometric linearization theorem for nondegenerate singularities of integrable systems 
\cite{Zung-Nondegenerate2012,Zung-SmoothLinearzation2012}, there is a smooth coordinate system $(x,y)$
in a neighborhood of $p$, in which we have the following normal form:
\begin{equation} \label{eqn:Period1}
X = \frac{2\pi}{f(x^2+ y^2)}(x\frac{\partial}{\partial y} - y\frac{\partial}{\partial x})
\end{equation}
where $f$ is a smooth function such that $f(0)\neq 0$.

\begin{remark}
Elliptic singular points are also called {\bf centers} in the literature. 
There is a nice similar normal form result of Maksymenko \cite{Maksymenko-Symmetries2010}
without the assumption that there is a first integral, but under the assumption that the orbits are closed.  
\end{remark}

A local coordinate system in which $X$ has normal form like the above form will be 
called a {\bf canonical coordinate system}. The function $f$ in Formula \eqref{eqn:Period1}
is the local \emph{period function} of the singular point: the orbit through each
point $(x,y)$ near $p$ is closed and is of period $f(x^2+y^2)$.  The function $f(x^2+y^2)$ 
projects to a smooth function on the local base space $\cB_{\cU(p)} = proj (\cU(p)),$ 
where $\cU(p)$ denotes a saturated tubular neighborhood of $p$ and $proj$ 
is the projection map, which will be denoted by $f_\cB$ and also called the local {\bf period function}
(on the local base space). In local canonical coordinates, the function $R = x^2 + y^2$ descends to a function on the base
space   $\cB_{\cU(p)}$  and becomes a smooth coordinate function there:
any other smooth function on $\cB_{\cU(p)}$ (i.e. whose pull-back by $proj$ is smooth on $\cU(p)$) is a smooth
function of $R$.  In particular, the period function $f_\cB$ is a smooth function on the local base space
$\cB_{\cU(p)}$ and is a function of $R$: $f_\cB = f \circ R.$

It is clear that any singularity of Type I is locally orbitally equivalent to a standard linear center. 
The period function is a local invariant of the singular point, because any isomorphism will preserve
the periods. More precisely, we have:

\begin{prop} Two singularities of Type I of two smooth integrable systems $(X_1,\cF_1)$ and $(X_2,\cF_2)$
are locally isomorphic if and only if the two corresponding period functions $f_{\cB_1}$ and $f_{\cB_2}$ 
on the two associated local base spaces $\cB_{1}$ and $\cB_2$ are left-equivalent, i.e. there is a local
smooth diffeomorphism $\Phi$ from $\cB_1$ to $\cB_2$ such that $f_2 = f_1 \circ \Phi.$ In terms of local
canonical coordinates, two integrable vector fields
\begin{equation}
 X_1 = \frac{2\pi}{f_1(x_1^2+ y_1^2)}(x_1\frac{\partial}{\partial y_1} - y_1\frac{\partial}{\partial x_1})
\end{equation}
and
\begin{equation}
 X_2 = \frac{2\pi}{f_2(x_2^2+ y_2^2)} (x_2\frac{\partial}{\partial y_2} - y_2\frac{\partial}{\partial x_2})
\end{equation}
near two singular points of Type I are locally isomorphic if and only if there is a local smooth 
diffeomorphism $\phi$ from $(\bbR_{+}, 0)$ to itself such that $f_1  = f_2 \circ \phi.$
\end{prop}

\begin{proof}
It is clear that the left equivalence of the two period functions is a necessary condition for the existence of an isomorphism,
because an isomorphism between two vector fields must preserve the periods of periodic orbits. Conversely, assume that
$f_1 = f_2 \circ \phi$ where $\phi: \bbR_+ \to \bbR_+$ is a local smooth diffeomorphism.  Then it is easy to check that the map
\begin{equation}
 (x_1,y_1) \mapsto (x_2,y_2) = (\sqrt{\frac{\phi(x_1^2+y_1^2)}{ x_1^2 + y_1^2}} x_1, \sqrt{\frac{\phi(x_1^2+y_1^2)}{ x_1^2 + y_1^2}} y_1)
\end{equation}
is a smooth map which sends $X_1$ to $X_2.$
\end{proof}

\subsection{Hyperbolic singularities with eigenvalue 0 (Type II)} \label{Section:hyperbolic0} 
A singular point $p \in \Sigma$ of a smooth integrable system $(X, \cF)$ on $\Sigma$ is called 
{\bf nondegenerate hyperbolic singular point with eigenvalue 0}, or also 
a singular point of {\bf Type II} if it satisfies the 
following conditions:

i) The linear part of $X$ at $p$ has two different real eigenvalues: $\gamma_1 =0$ and $\gamma_2 \neq 0$.

ii) There is $F \in \cF$ which is not flat at $p$, i.e. its Taylor series is non-trivial.

According to the local geometric linearization theorem \cite{Zung-Nondegenerate2012,Zung-SmoothLinearzation2012}, 
for each nondegenerate
hyperbolic singular point $p$ with eigenvalue 0, there is a local smooth coordinate system $(x,y)$ in a neighborhood 
$\cU$ of $p$ such that $X$ has the normal form
\begin{equation} \label{eqn:NF2}
 X = \gamma(y)x\frac{\partial}{\partial x},
\end{equation}
where $\gamma(0) \neq 0$, and $F$ is a function of $y$. In particular, the set of singular points of $X$ near $p$
is a smooth curve $\cS$ given by the equation
\begin{equation}
 \cS = \{q \in \cU \ |\ X(q) = 0\} = \{q \in \cU \ |\ x(q) = 0\},
\end{equation}
and for each point $q \in \cS$, the eigenvalue of $X$ at $q$ is equal to $\gamma(y(q))$. Since the eigenvalue does not 
depend on the choice of coordinates, the eigenvalue function on $\cS$,
\begin{equation}
 q \in \cS \mapsto f(y(q))
\end{equation}
  is a local invariant of $X$ at $p$. We can formulate this fact as the following proposition, whose proof is straightforward.

\begin{prop} \label{prop:TypeII}
 Two smooth integrable systems $(X_1,\cF_1)$ and $(X_2,\cF_2)$ are locally smoothly isomorphic near 
two respective  singular points $p_1$ and $p_2$ of Type II if and only if there is a local eigenvalue-preserving diffeomorphism 
from the curve $\cS_1$ of singular points of $X_1$ near $p_1$ to the curve $\cS_2$  of singular points of $X_2$ near $p_2$. 
In terms of local normal forms, $X_1 = \gamma_1(y)x\frac{\partial}{\partial x}$ is locally smoothly isomorphic to 
$X_2 = \gamma_2(y)x\frac{\partial}{\partial x}$ if and only if $\gamma_1$ is left equivalent to $\gamma_2$, 
i.e. there is a local diffeomorphism $\psi: (\bbR, 0) \to (\bbR, 0)$
such that $\gamma_1(y) = \gamma_2(\psi(y))$.
\end{prop}

Observe that the function $f(y) = 2\pi  \sqrt{-1}/ \gamma(y)$, where $\gamma(y)$ is the  \emph{eigenvalue function}  
in the local normal form \eqref{eqn:NF2} can also be interpreted as a local
\emph{period function} as follows: 
Complexify the system in the coordinate $x$ (so $x$ is now a complex coordinate, and $y$ is still
a real coordinate, the manifold has 1 complex dimension plus 1 real dimension). 
Then the flow of $X$ on each local complex line $\{y= const\}$ is periodic in imaginary time and is of period
equal to $f(y).$ So we will call $f(y)$ the (imaginary) local {\bf period function} of $X$ near a singular point of Type II. 
Proposition \ref{prop:TypeII} can be paraphrased as follows: Type II singular points are classified up to 
local isomorphisms by the left equivalence class of their imaginary period functions.

\subsection{Hyperbolic singularities without eigenvalue 0 (Type III)}

To say that $p$ is a singular point of Type III of $(X,\cF)$ means that  
$X$ has two non-zero real eigenvalues at $p$, and there is a first integral which is not flat at $p$.
According to the local geometric linearization theorem 
\cite{Zung-Nondegenerate2012,Zung-SmoothLinearzation2012}, $X$ has the following  normal form:
\begin{equation} \label{eqn:NF_III}
X = h(x^ay^b)(\frac{x}{a}\frac{\partial}{\partial y} -\frac{y}{b}\frac{\partial}{\partial x})
\end{equation}
where $a, b \in \bbN$ are coprime, and $h$ is a smooth function such that $h(0) \neq 0$. 

In the above canonical coordinates, the function $x^ay^b$ generates the ring  local first integrals of $X$.
More precisely, if $G$ is a smooth local first integral, 
then in each local quadrant $\{\epsilon x \geq 0, \delta  y \geq 0\}$,
where $\epsilon = \pm$ and $\delta = \pm$, we can write
\begin{equation}
G (x,y) = g^{\epsilon, \delta} (h(x^ay^b)) 
\end{equation}
where $g^{+,+},g^{+,-},g^{-,+},g^{-,-}$ are smooth functions which have the same Taylor series at 0.
The proof of this fact follows easily from formal computations.

Similarly to the Type II case, $X$ is of (complex) toric degree 1 (see \cite{Zung-Convergence2002} for the notion
of toric degree of a vector field at a singular point), and the $\bbT^1$-action in the complexified space associated to 
$X$ is generated by the linear vector field 
\begin{equation}
Y = 2 \pi  \sqrt{-1} (\frac{x}{a}\frac{\partial}{\partial y} -\frac{y}{b}\frac{\partial}{\partial x}) 
\end{equation}
in canonical coordinates. In the analytic case, $Y$ is uniquely determined by $X$ (see \cite{Zung-Convergence2002}), 
though in the smooth case it is only unique up to a flat term. So any automorphism or isomorphism of $X$ will also
be an automorphism or isomorphism of $Y$ up to a flat term. Thus, in order to understand the local automorphisms of $X$,
we need to understand the automorphisms of the linear vector field $Y$ (or $Y / (2 \pi \sqrt{-1})$).

 The function $h(x^ay^b)$ in the normal form \eqref{eqn:NF_III} 
is directly related to the local period function of $X$ in the complexified space
(at least in the analytic case) : The flow of the  vector field 
$X = h(x^ay^b)(\frac{x}{a}\frac{\partial}{\partial y} -\frac{y}{b}\frac{\partial}{\partial x})$ in $\bbC^2$ in the imaginary time
will have closed orbits with periods equal to $\frac{2 \pi a b \sqrt{-1}}{h(x^ay^b)}.$
We will call $h(x^ay^b)$ the {\bf frequency function} of $X$ near $p$.

\begin{prop} \label{eqn:preservesX1} Let $a$ and $b$ be two coprime natural numbers. 

1) A local real analytic diffeomorphism $(x,y) \mapsto (x_1,y_1)$ preserves the linear vector field 
$\displaystyle \frac{x}{a}\frac{\partial}{\partial x} - \frac{y}{b}\frac{\partial}{\partial y}$ if and only if 
there are two local analytic functions $\rho_1, \rho_2$ of one variable such that $\rho_1(0) \neq 0, \rho_2(0) \neq 0$ and
\begin{equation}
\begin{cases}
 x_1 = x\rho_1(x^ay^b)\\
y_1 = y\rho_2(x^ay^b)
\end{cases}.
\end{equation}

2) A local smooth diffeomorphism $(x,y) \mapsto (x_1,y_1)$ preserves the linear vector field 
$\displaystyle  \frac{x}{a}\frac{\partial}{\partial x} - \frac{y}{b}\frac{\partial}{\partial y}$
if and only if 
there are four couples of local smooth functions $\rho_1^{\varepsilon,\delta}, \rho_2^{\varepsilon,\delta}$ 
where $\varepsilon = \pm$ and $\delta = \pm$, which do not vanish at 0, such that the Taylor series 
of  $\rho_1^{\varepsilon,\delta}$ and $\rho_2^{\varepsilon,\delta}$ do not depend on $\varepsilon$ and $\delta$, and such that
\begin{equation}
\begin{cases}
 x_1 = x\rho_1^{\varepsilon,\delta} (x^ay^b)\\
y_1 = y\rho_2^{\varepsilon,\delta}(x^ay^b)
\end{cases}
\end{equation}
if $\varepsilon x \geq 0, \delta y \geq 0$.
\end{prop}

\begin{proof}
 1) Let $(x,y) \mapsto (x_1,y_1)$  be a local real analytic diffeomorphism which preserves the linear vector field
 $X^{(1)} = \frac{x}{a}\frac{\partial}{\partial x} - \frac{y}{b}\frac{\partial}{\partial y}$. Since $\{ x = 0\}$ (resp. $\{ x = 0\}$)
is the stable (resp. unstable) manifold of $X^{(1)}$, we must have $x_1= 0$ on $\{ x = 0\}$ and $y_1= 0$ on $\{ y = 0\}$. In other words,
we can write $x_1 = x.\theta_1(x,y), y_1 = y.\theta_2(x,y)$ where $\theta_1, \theta_2$ are two analytic functions. Note that the time-$t$
flow of the linear vector field $X^{(1)}$ multiplies $x$ by $e^{t/a}$, and also multiplies $x_1$ by $e^{t/a}$ because the map 
$(x,y) \mapsto (x_1,y_1)$ preserves $X^{(1)}$. Therefore the quotient $\frac{x_1}{x}=\theta_1(x,y)$ is invariant by the flow of  $X^{(1)}$.
In other words, $\theta_1(x,y)$ is a first integral of  $X^{(1)}$. Any local analytic first integral $\theta_1(x,y) = \rho_1(x^ay^b)$.
Similarly, we have $\theta_2(x,y) = \rho_2(x^ay^b)$.

Conversely, it is easy to see that the map $(x,y) \mapsto (x\rho_1(x^ay^b),y\rho_2(x^ay^b))$ preserves 
$X^{(1)} = \frac{x}{a}\frac{\partial}{\partial x} - \frac{y}{b}\frac{\partial}{\partial y}$.

2) The proof in the smooth case is similar, except that we must write $\theta_1(x,y) = \rho_1^{\varepsilon,\delta}$
dependent on the quadrant in $\bbR^2$, where $\rho_1^{+,+}, \rho_1^{+,-}, \rho_1^{-,+}$ and $\rho_1^{-,-}$ are four functions
which have the same Taylor series at 0. (Actually, the number of functions can be reduced from 4 to 2, because, for example, if $a$ is odd 
then $\rho_1^{+,\delta}$ and $\rho_1^{-,\delta}$ can be chosen to be the same function).
\end{proof}

\begin{prop} \label{prop:TypeIII_local}
Let $a$ and $b$ be two coprime natural numbers.

1) Two local real analytic vector fields 
$X_1 = h_1(x_1^ay_1^b)\big(\frac{x_1}{a}\frac{\partial}{\partial x_1} - \frac{y_1}{b}\frac{\partial}{\partial y_1}\big)$
and $X_h = h_2(x_2^ay_2^b)\big(\frac{x_2}{a}\frac{\partial}{\partial x_2} - \frac{y_2}{b}\frac{\partial}{\partial y_2}\big)$
are locally analytically isomorphic if and only if $h_1$ and $h_2$ are left equivalent, i.e. we can write $h_1 = h_2\circ \psi$
where $\psi: (\bbR, 0) \to (\bbR, 0)$ is analytic with $\psi'(0) \neq 0$.

2) In the smooth case, when $h_1$ and $h_2$ are smooth functions, then $X_1$ and $X_2$ are locally smoothly isomorphic if 
and only if $h_1$ and $h_2$ are formally left-equivalent, i.e. 
\begin{equation}
Taylor(h_1) = Taylor(h_2)\circ \psi 
\end{equation}
where $\psi$ is a formal series with $\psi(0) = 0, \psi'(0) \neq 0$, and $Taylor(h)$ means the Taylor series of $h$.
\end{prop}

\begin{proof}
1) The toric degree of $X_1$ (in the sense of \cite{Zung-Convergence2002}) is 1, and the corresponding $\bbT^1$-action 
in the complexified space which preserves $X$ is generated by 
$\frac{x_1}{a}\frac{\partial}{\partial x_1} - \frac{y_1}{b}\frac{\partial}{\partial y_1}$.  
Thus if there is a local analytic diffeomorphism $\Phi$ such that $\Phi_*X_1 = X_2$, then we also have 
$\Phi_*\big(\frac{x_1}{a}\frac{\partial}{\partial x_1} - \frac{y_1}{b}\frac{\partial}{\partial y_1}\big) = 
\big(\frac{x_2}{a}\frac{\partial}{\partial x_2} - \frac{y_2}{b}\frac{\partial}{\partial y_2}\big)$. 
According to Proposition \ref{eqn:preservesX1}, 
after the diffeomorphism $\Phi$ we can write 
\begin{equation}
\begin{cases}
 x_2 = x_1\rho(x^ay^b)\\
y_2 = y_1\theta(x^ay^b)
\end{cases}.
\end{equation}
In particular, $x_2^ay_2^b = (x_1^ay_1^b).\rho(x_1^ay_1^b).\theta(x_1^ay_1^b) = \psi(x_1^ay_1^b)$ where $\psi(z) = z\rho(z)\theta(z)$.
We also have $\Phi_*h_1 =h_2$, therefore $h_2 = h_1 \circ \psi$.

Conversely, assume that $h_2 = h_1 \circ \psi$. Then it is easy to check that the map $(x_2,y_2) = \Phi(x_1,y_1)$ given by the formula 
\begin{equation}
\begin{cases}
 x_2 = x_1\\
y_2 = y_1\sqrt[b]{\frac{\psi(x_1^ay_1^b)}{x_1^ay_1^b}}
\end{cases}
\end{equation}
(this map is well-defined because $\frac{\psi(z)}{z} \neq 0$ when $z = 0$) sends $X_1$ to $X_2$.

2) The ``only part'' in the smooth case is absolutely similar to the analytic case. The ``if'' part follows from Sternberg-Chen's theorem
\cite{Chen-Equivalence1963, Sternberg-structure1958} which says that if two smooth hyperbolic vector fields are formally isomorphic then they are locally smoothly isomorphic.
\end{proof}

\section{Semi-local structure of nondegenerate singularities}

\subsection{Level sets with singular points of Type II} \label{subsection:semilocal_II}
Denote by $\cN$ a level set of a weakly nondegenerate smooth integrable system $(X,\cF)$ on a compact surface $\Sigma$, 
which contains a singular point of Type II, but does not contain singular points of the other Type I, III, IV. As will be shown below, 
$\cN$ is a smooth circle. A tubular neighborhood of $\cN$ will be either orientable (a cylinder) or non-orientable (a Mobius band).
The case when it is a Mobius band will be called the {\bf twisted case}, and the case when it is a cylinder will be
called the {\bf non-twisted case}. Notice that if $\Sigma$ is orientable then we only have the non-twisted case. 

\begin{prop} \label{prop:TypeII_semilocal}
 With the above notations and assumptions, we have:

i) $\cN$ is a smooth circle.

ii) $\cN$ contains an even number $2m>0$ of singular points of Type II.

iii) In the non-twisted case there is a tubular neighborhood $\cU(\cN)$ of $\cN$ such that the projection map
\begin{equation}
proj : \cU(\cN) \to \cB
\end{equation}
from $\cU(\cN)$ to the base space $\cB$ of the associated singular fibration is a smooth trivial circle fibration over its image
\begin{equation}
 \cB_{\cU(\cN) } := proj(\cU(\cN)), 
\end{equation}
and the set of singular points  $\cS = \{q \in \cU(\cN) \ |\ X(q) = 0\}$ is a disjoint union
\begin{equation}
 \cS = \bigsqcup_{i=1}^{2m}\cS_i
\end{equation}
of $2m$ smooth sections of the circle fibration $\cU(\cN) \to  \cB_{\cU(\cN)}$, and every point of $\cS$ is singular of Type II. 

iv) In the twisted case, there is still a tubular neighborhood  $\cU(\cN)$ of $\cN$ which is saturated and smoothly foliated by the
level sets of $(X,\cF)$, and $\cN$ is the only exceptional leaf of this foliation in $\cU(\cN)$ (the one with non-trivial holonomy),
the set of singular points $\cS = \{q \in \cU(\cN) \ |\ X(q) = 0\}$ is still a disjoint union 
$\displaystyle \cS = \bigsqcup_{i=1}^{2m}\cS_i$ of an even number of smooth curves.

v) In the non-twisted case, $\cN$ is a regular level set of the associated fibration, i.e.
there is $F \in \cF$ which is regular at $\cN$. In the twisted case, $\cN$ is a singular level set of Morse-Bott type,
i.e. there is $F \in \cF$ such that $F$ is of the type $F=z^2$ in the neighborhood of every point of $\cN$.  
\end{prop}

\begin{proof}
i) Denote by $\cK$ smallest closed invariant set of the vector field $X$ which contains $p$ and satisfies
the following additional property: If $q \in \cK$ and $\cO$ is an orbit of $X$ such that its 
closure contains $q$, then $\cO \subset \cK$. (Such a set $\cK$
exists because it is the intersection of all closed invariant sets which satisfy these properties).
It follows immediately from the definition of $\cK$  that $\cK$ is connected.
By continuity, any first integral $F \in \cF$ is constant on $\cK$, so we have $\cK \subset \cN$.

We will show that in fact $\cK$ is a smooth circle and $\cN = \cK$.

Since $\cK \subset \cN$, all singular points of $X$ in $\cK$ are of type II by our assumptions. Remark that, if $q$ is a singular 
point of type II, then there are exactly two regular orbits of $X$ which contain $q$ in their closure, and moreover their union
forms together with $q$ a curve which is smooth at $q$: in local normal form  $X = f(y)x\frac{\partial}{\partial x}$ they are given by 
$\{x<0, y = 0\}$ and $\{x>0, y = 0\}$. Remark also that the number of singular points in $\cK$ is finite, otherwise 
there would exist an accumulation point $q_\infty \in \cK, q_\infty = \lim_{n \to \infty}q_n, q_n \neq q_\infty, q_n \in \cK$ 
singular of Type II, and then any smooth first integral will be constant on $\{q_n, n\in \bbN \} \subset \cK$, which implies that it is 
flat at $q_\infty$, which is a contradiction to our definition of type II singular points.

Starting from the point $p$, take the two regular orbits whose closures contain $p$, then take the singular points in the closure
of these orbits, then take the regular orbits whose closures contain these new singular points, and so on. By definition,
all these orbits and singular points belong to $\cK$. The process must stop after a finite number of steps, because $\cK$
contains a finite number of singular points. This process gives us a smooth curve, so $\cK$ is closed and smooth, i.e. it is a circle.

By our definition of Type II singularities, there is a smooth first integral $F$, which is constant on $\cK$ of course, and
which is not flat at $p$. It follows easily that any point outside of $\cK$ can be separated from $\cK$ by a first integral. In other 
words, $\cK$ is a level set, i.e. $\cK = \cN$, and therefore $\cN$ is a smooth circle.

ii) This statement is a particular simple case of our study of nondegenerate $\bbR^n$-actions on $n$-manifolds in
\cite{ZungMinh-Action2012} (see Subsection 3.2 of \cite{ZungMinh-Action2012}).
\begin{figure}[htb] 
\begin{center}
\includegraphics[width=50mm]{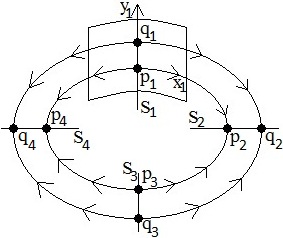}
\caption{Trivial holonomy.}
\label{fig:Trivial-holonomy}
\end{center}
\end{figure}

iii) Consider the non-twisted case. Denote by $p_1 = p, p_2, \hdots, p_{2m}$ the singular points of Type II of $X$
on $\cK$ in a cyclic order. Near each $p_i$ fix a canonical coordinate system $(x_i, y_i)$ in a neighborhood $\cU_i$
of $p_i$ in which $X = f_i(y_i)x_i\frac{\partial}{\partial x_i}$, and such that these coordinate systems give
the same orientation on $\cK$ and in a neighborhood of $\cK$. 
Denote by $\cS_i = \cU_i \cap \{x_i = 0\}$ the curve of singular points of Type II passing through $p_i$.

Take a point $q_1 \in \cS_1, q_1 \neq p_1$, $q_1$ close enough to $p_1$. Then the set $\{y_1 = y_1(q_1), x_1 > 0\}$ 
is a local regular orbit which tends to $q_1$ in one direction. By continuity, this orbit remains close to $\cK$
and enters a small neighborhood $\cU_2$ of $p_2$ in the other direction, and so it tends to some point $q_2 \in \cS_2$ 
close to $p_2$. Similarly, there is a regular orbit of $X$ which tends to $q_2$ in one direction and tends to a point 
$q_3 \in \cS_3$ in the other direction, and so on.  Finally, there is a regular orbit which tends to $q_{2m} \in S_{2m}$
in one direction, and tends to a point $q_1' \in \cS_1$ in the other direction. So we get a smooth map from $\cS_1$ to 
itself defined by $q_1 \mapsto q_1'$, which may be called the {\bf holonomy} of $X$ along $\cK$.
Remark that, in the non-twisted case, $y_i(q_i)$ have the same sign for all $i$, and in particular 
$y_1(q_1)$ and $y_1(q_1')$ have the same sign. Due to the existence of a smooth global first integral which is 
non-flat for every singular point, it is easy to see that the holonomy is actually trivial, i.e. 
$q_1 = q_1'$. Indeed, if $y_1(q_1)> y_1(q_1')>0$ for example, then we can repeat this process (iterate the holonomy map)
to get a sequence of points $q_1, q_1', q_1'', q_1''', \hdots, q_1^{(n)}, \hdots$, which must tend to some 
point $q_1^\infty \in \cS_1$, and any global smooth integral will be flat at $q_1^\infty$, 
which contradicts our assumptions about the singular points of $(X,\cF)$.
(In this proposition, we really need the existence of global non-flat first integrals, and not just local non-flat
 first integrals, otherwise there will be counter-examples). Since the holonomy is trivial, we have a smooth 
 foliation of a neighborhood of $\cN$ into circles with trivial holonomy. The rest of the proof is straightforward.     

iv) The proof of Assertion iv) is similar to the proof of Assertion iii).

v) The proof is straightforward.
\end{proof}
\begin{figure}[htb] 
\begin{center}
\includegraphics[width=40mm]{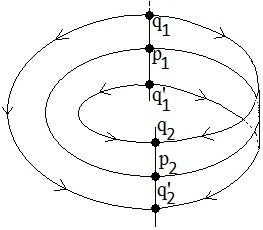}
\caption{Mobius band in the twisted case.}
\label{fig:Mobius-Twistedcase}
\end{center}
\end{figure}

Observe that, on each curve $\cS_i$ of singular points in a tubular neighborhood $\cU(\cN)$ of $\cN$, we have an eigenvalue function, 
whose value at each point $q \in \cS_i$ is the non-zero eigenvalue of $X$ at $q$.
In the non-twisted case, these functions may be viewed as functions on the local base space $\cB_{\cU(\cN)} = proj(\cU(\cN))$ via the projection 
map, which we will denote by $E_1, \hdots, E_{2m}$:
\begin{quote}
For $\xi \in \cB_{\cU(\cN)}, E_i(\xi)$ is the eigenvalue of $X$ at $proj^{-1}(\xi) \cap \cS_i$.
\end{quote}
Of course, these eigenvalue functions on $\cB_{\cU(\cN)}$ are invariants of $X$ in a neighborhood $\cU(\cN)$ of $\cN$. 
In the twisted case, we still have $2m$ eigenvalue functions, but they don't descend to functions on $\cB_{\cU(\cN)}$ in general, only 
to a branched 2-covering of $\cB_{\cU(\cN)}$.

Besides the eigenvalue functions, which are invariant of local character of the vector field $X$, 
there is another invariant of $X$ in $\cU(\cN)$, of a more global character, which was discovered in \cite{ZungMinh-Action2012}
and called the {\bf monodromy}. The monodromy is defined as follows:

Let $M$ be an closed invariant curve (level set) in $\cU(\cN)$, and denote by $q_1, \hdots,q_{2n}$ the singular points of Type II on $M$
in a cyclic order. (if $\cU(\cN)$ is twisted and $M \neq \cN$ then $n= 2m$, otherwise $n=m$). As was observed in \cite{ZungMinh-Action2012},
for each $q_i$ there is an involution (reflection) which preserves $X$ and exchanges the regular orbit on the left of $q_i$ with the orbit on the right of $q_i$.
Take a regular point $z_0$ which lies between $q_{2n}$ and $q_1$, and put $z_1=\sigma_1(z_0),z_2=\sigma_2(z_1),\hdots, z_{2m}=\sigma_{2m}(z_{2m-1})$.
Then $z_{2m}$ and $z_0$ lies on the same orbit of $X$, and there is a unique number $\mu = \mu(M)$ such that
\begin{equation}
\varphi_X^\mu (z_0) = z_{2m}.
\end{equation}
It was observed in \cite{ZungMinh-Action2012} that $\mu$ depends only on $X$, $M$ and the choice of the orientation on $M$: when the orientation
is inversed the $\mu$ changes to $-\mu$.

Since $\mu$ can be defined for each level set $M$ in $\cU(\cN)$, we get a map from $\cB_{\cU(\cN)}$ to $\bbR$, which associates to each element $\xi \in \cB_{\cU(\cN)}$ the monodromy of $X$ on $proj^{-1}(\xi)$ (with respect to a given choice of orientation).

\begin{prop} \label{prop:invariantTypeII}
The topological type of $\cU(\cN)$ (twisted or non-twisted), the number of singular points of Type II on $\cN$, the eigenvalue functions, and the
monodromy function form together a complete set of invariants of a weakly nondegenerate smooth integrable system $(X,\cF)$ in the neighborhood of a 
compact level set $\cN$ which contains only singular points of Type II. 
\end{prop}
\begin{figure}[htb] 
\begin{center}
\includegraphics[width=40mm]{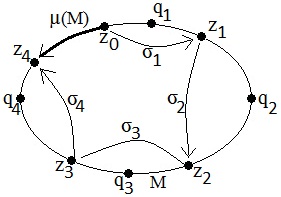}
\caption{Monodromy}
%\label{fig:Mobius-Twistedcase}
\end{center}
\end{figure}
It means that $(X_1,\cF_1)$ in $\cU(\cN_1)$ is semi-locally 
isomorphic to $(X_2,\cF_2)$ in $\cU(\cN_2)$, where $\cN_1$ and $\cN_2$ are level sets of Type II, if and only if 
there is a semi-local fibration-preserving diffeomorphism $\Phi$ from $\cU(\cN_1)$ to $\cU(\cN_2)$ which is 
bijection between the sets of Type II singular points, and which also preserve the monodromy function and the eigenvalue functions.
\begin{proof}
The proof in the non-twisted case is just a parametrized version of a result of \cite{ZungMinh-Action2012} 
(see Theorem 3.7 of \cite{ZungMinh-Action2012}). The twisted case can be reduced to the non-twisted case by taking a double covering.
\end{proof}

\begin{remark}
Proposition \ref{prop:invariantTypeII} remains true when $\cN$ does not contain any 
singular point of $X$ at all, and in that case the monodromy function is nothing but
the period function, i.e. the time it takes for the flow to go full circle on closed orbits.
\end{remark}

\subsection{Level sets with singular points of Type III}
\begin{defn}
We will say that a level set $\cN$ of a smooth integrable system $(X, \cF)$ is a {\bf singular level set of Type III}
if $\cN$ contains at least one singular point of Type III, and any singular point of $X$ in $\cN$ is of Type II or Type III.
\end{defn}
\begin{prop} \label{prop:TypeIII_1}
Let $\cN$ be a singular level set of Type III of integrable system $(X, \cF)$ on a compact surface 
$\Sigma$. Then we have:

i) Topologically, $\cN$ is a connected finite graph, whose edges are regular orbits of $X$,
and whose vertices are singular points of Type II or Type III: Each singular point of Type II is a vertex of 
valency 2, and each singular point of Type III is a vertex of valency 4.
Moreover, $\cN$ is a finite union $\bigcup_{k=1}^{s}\cS_k$ of $s$ smooth circles $\cS_k$ ($s \geq 1$) 
with transversal (self-)intersections at singular points of Type III.

ii) There exists a smooth first integral $F \in \cF$ such that $F=0$ on $\cN$ and the multiplicity of $F$
at each $\cS_k$ is a natural number $m_k$, such that the greatest common divisor of $(m_1, \hdots, m_s)$
is 1 or 2, and such that any other first integral $G \in \cF$ can be written as 
\begin{equation}
 G = g(F) + G_{flat},
\end{equation}
where $g$ is a smooth function and $G_{flat}$ is flat at $\cN$. If  
$gcd(m_1,\hdots, m_s) = 2$, then $F$ can be chosen to be a non-negative function. 
\end{prop}

\begin{proof}
i) The proof of Assertion i) is similar to the proof of Assertion i) of Proposition \ref{prop:TypeII_semilocal}.

ii) Similarly to the proof of Assertion iii) of Proposition 
\ref{prop:TypeII_semilocal}, take a point $q$ on a local curve $L$ which intersects $\cN$
transversally at a regular point $q_0$, and let it go by the flow of $X$ and jump over the points of Type II
whenever the flow tends to such a point. Then we will get a path which moves along some components of $\cN$
and then returns to a point $q'$ on $L$ after going around. In the non-twisted case, when $\cU(\cN)$ is orientable,
then $q$ and $q'$ lie on the same side with respect to $q_0$ on $L$. In the twisted case, the first return point $q'$
may lie on the opposite side of $q$ with respect to $q_0$ on $L$. If $q'$ lies on the opposite side of $q$, then we
continue the path until we return to $L$ again, this time at a point $q"$ which lies on the same side
as $q$. Using arguments similar to the ones in the proof of Proposition \ref{prop:TypeII_semilocal}, one can show that this return 
map must in fact be the identity map, i.e. we have either $q = q'$ or $q=q"$.

It follows from the above arguments that $\cU(\cN)\setminus \cN$ is foliated by smooth invariant circles, and 
these circles are also regular level sets of $(X,\cF)$.

Take an arbitrary smooth first integral $G \in \cF$ which is non-flat at $\cN$.  Then $G(\cS_k) = 0$
for each circle $\cS_k$ in $\cN$, and moreover the multiplicity (i.e. order of vanishing) of $G$ on $\cS_k$
is a finite number $M_k < \infty$.

If $p_i \in \cS_j \cap \cS_k$ with a local canonical coordinate system $(x_{ij}, x_{ik})$ such that 
\begin{equation}
X = f_i(x_{ij}^{\alpha _{ij}}x_{ik}^{\alpha _{ik}})
\big( \frac{x_{ij}}{\alpha _{ij}}\partial x_{ij} - \frac{x_{ik}}{\alpha _{ik}}\partial x_{ik}  \big),
\end{equation} 
and $\alpha _{ij}, \alpha _{ik} \in \bbN, gcd(\alpha _{ij}, \alpha _{ik}) =1$, then the smoothness of
$G$ at $p_i$ implies that 
\begin{equation}
\frac{M_k}{\alpha _{ik}} = \frac{M_j}{\alpha _{ij}} \in \bbN.
\end{equation} 
Notice that, if $\frac{M_k}{\alpha _{ik}} = \frac{M_j}{\alpha _{ij}}$ then it automatically means that
this fraction is a natural number, because $\alpha _{ij}$ and $\alpha _{ik}$ are coprime. Put 
\begin{equation}
m_k = \frac{M_k}{D},
\end{equation} 
where $D=gcd(M_1, \hdots, M_s)$ is the greatest common divisor of $M_1, \hdots, M_s$.
Then we still have $\frac{m_k}{\alpha _{ik}} = \frac{m_j}{\alpha _{ij}} \in \bbN$ for every singular
point $p_i$ of Type II ($p_i \in \cS_j \cap \cS_k$). It is easy to see that the function
\begin{equation}
 \sqrt[D]{G^2}
\end{equation}
is a well defined smooth first integral whose order of vanishing on $\cS_k$ is 
$2m_k$ for all $k = 1, \hdots, s$. Either this function, or its square root $\sqrt[D]{G}$  if a smooth 
single-valued function $\sqrt[D]{G}$ can be defined in a neighborhood of $N$, will be the required first integral.
\end{proof}

\begin{remark}
In the above proposition, even in the non-twisted (i.e. when $\cU(N)$ is orientable) we may have $gcd(m_1,\hdots, m_s) = 2,$
while even in the twisted case we have have $gcd(m_1,\hdots, m_s) = 1$, as the examples 
pictured in Figure \ref{figure:gcd} show. 
\end{remark}

\begin{figure}[htb] 
\begin{center}
\includegraphics[width=120mm]{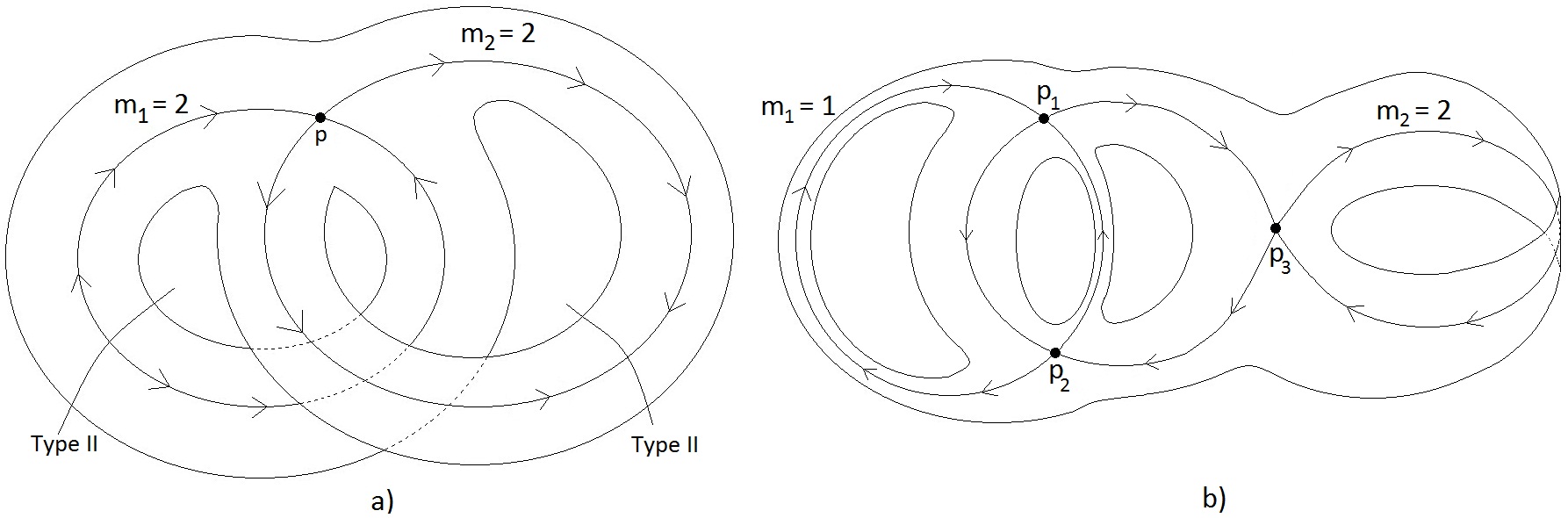}
\caption{a) Non-twisted example with $gcd(m_1,\hdots, m_s) =2$; b) Twisted example with $gcd(m_1,\hdots, m_s) =1$}
 \label{figure:gcd}
\end{center}
\end{figure}

The semi-local first integral function $F$ given by the above proposition can be viewed as a function on the local base space
$\cB_{\cU(N)}$, and will be called a local {\bf coordinate function} on $\cB_{\cU(N)}$: if $G$ is any other smooth function
on $\cB_{\cU(N)}$ then there exists a family of smooth real functions $g_i$ of one variable, one for each local edge of $\cB_{\cU(N)}$
(when $\cB_{\cU(N)}$ is a local graph with one vertex and more than one edges), such that $G = g_i(F)$ on each edge, and
all the functions $g_i$ have the same Taylor expansion (i.e. the diffrence of any two of them is a flat function).

\begin{prop}
Two smooth integrable systems $(X_1,\cF_1)$ and $(X_2,\cF_2)$ near two respective singular level sets of Type III $N_1$
and $N_2$ are semi-locally orbitally equivalent if and only if there is a homeomorphism from a neighborhood of  $N_1$ 
to a neighborhood of $N_2$, which sends $N_1$ to $N_2$, singular points of Type II on $N_1$ to singular points of Type II
on $N_2$, and preserves the ratio $-a:b$ of the two eigenvalues of each singular point of Type III.
\end{prop}

\begin{proof}
 It follows easily from Proposition \ref{prop:TypeIII_1}.
\end{proof}

\subsection{The period cocycle}

In order to classify singular level sets of Type III semi-locally, we need an additional invariant, which is the cohomology class
of the period cocycle defined below.  The construction is similar to the one introduced in \cite{DufourMolino-Class2dim1994}
for obtaining symplectic invariants of Hamiltonian systems, though our situation is more complicated.

Let us fix a semilocal first integral $F$ in $\cU(N)$  with lowest multiplicity, as given by Assertion ii) of Proposition \ref{prop:TypeIII_1}.
According to Proposition \ref{prop:TypeIII_local}, for each singular point $p_i \in N$ of Type III, 
we can choose a local canonical coordinate system $(x_i,y_i)$ in which  the vector field $X$ has the form
\begin{equation} \label{eqn:NFIII_p_i}
 X =  h_i (x_i^{a_i} y_i^{b_i}) (\frac{x_i}{a_i}\frac{\partial}{\partial y_i} -\frac{y_i}{b_i}\frac{\partial}{\partial x_i})
\end{equation}

Take an edge $E$ in $N$. For simplicity, let us assume for the moment that $E$ does not contain singular points of Type II.
To fix the notations, assume that $p_i$ (resp. $p_j$) is the limit of the points of 
$E$ by the flow of the vector field $X$ in the negative (resp. positive) time direction, and that
near $p_i$ we have $E \supset \{y_i=0, x_i > 0 \}$ and   near $p_j$ we have $E \supset \{y_j=0, x_j  < 0 \}$. 
Denote by $A_E = \{x_i = 1\}$ and $B_E = \{x_j = -1\}$ the two local curves (in two local coordinate systems)
given by these equations. Then the flow of $X$ will take
each point of $A_E$ to a point of $B_E$ after some time. So we get a time function  for going from a point of $A_E$ to $B_E$
by $X$, which may be viewed as a function on $A_E$ (this function may admit any value, positive or negative).
Since    the multiplicity of $F$ at $E$ is equal to $m_E = m_{i-} = m_{j-}$, time function
we can view this function as a function of $F$, which is a-priori not regular in $F$ but regular in $F^{1/m_E}$. 
(In the analytic case, it would be a Puiseux series in $F$; in the smooth case we can still talk about its Puiseux series).
Let us denote this local function of one variable by $P_E$, i.e. the value of the function at each point $q \in A_E$ is $P_E(F(q))$. 
If $E$ contains points of Type II, then $P_E$ can be defined in the same way, by jumping over the points of Type II,
like we did in \cite{ZungMinh-Action2012} and in the previous subsection for the definition of monodromy.

\begin{figure}[htb]  \label{figure:PeriodCocycle}
\begin{center}
\includegraphics[width=80mm]{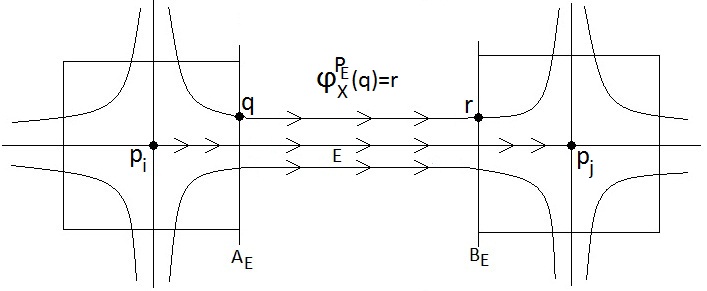}
\caption{Component $P_E$ of the period cocycle}
\end{center}
\end{figure}

Thus, for each   $E$ we get a function $P_E$. 
The family $(P_E, \ E\ \text{is an edge of}\ N)$ will be called the 
{\bf period cocycle}. It is easy to see that this period cocycle is arbitrary, i.e. any family of local 
functions $(P_E, \ E\ \text{is an edge of}\ N)$, where each $P_E(z)$ is a smooth function in $z^{1/m_E},$
can be realized by a smooth integrable system, by the gluing method. However, this cocycle 
is not an invariant of the system, because it depends on
the choice of $F$ and of local canonical coordinates. 
In order to get an invariant, we have to take its equivalence class with respect to a natural 
equivalence relation generated by 2 kinds of  operations: changing a canonical
coordinate system by another canonical coordinate system, and changing the first integral $F$ by another first integral.
Changing $F$ by another first integral (which has the same multiplicity at the components of $N$ as $F$) 
simply leads to the left equivalence (in the sense of left equivalence of maps), while changing local coordinates
leads to a cohomology class. So our invariant is the left equivalence class of a cohomology class.

Assume, for example, that the coordinate system $(x_1,y_1)$ in a neighborhood of
$p_1$ is replaced by another canonical coordinate system $(x_1',y_1')$ in the same neighborhood. According 
to Proposition \ref{prop:TypeIII_local}, we have
\begin{equation}
 x_1' = x_1 \rho^{\pm}(x_1^{a_1}y_1^{b_1}), y_1' = y_1 \theta^{\pm}(x_1^{a_1}y_1^{b_1})
\end{equation}
for some \emph{multi-branched} smooth functions $\rho^{\pm}, \theta^{\pm}$ such that $\rho(0) \neq 0, \theta(0) \neq 0.$ 
Here, a multi-branched smooth function is a finite family of functions (for example $\rho^+$ and $\rho^-$) 
which coincide up to a flat term at the point in question, i.e. all the branches have the same Taylor series (so that they can be glued together
to become a smooth function on a non-separated manifold or a Reeb graph). Denote by 
$E_1 = \{y=0, x >0\},  E_2 = \{y=0, x <0\}, E_3 = \{y >0, x = 0\}, E_4 = \{y <0, x = 0\}$ the four local edges of
$N$ having $p_1$ as a vertex. Then the functions $P_{E_1},\hdots,P_{E_4}$ will be changed by the following
rule under the above change of coordinates:
\begin{equation} \label{eqn:change_of_coordinates}
\left\{ 
 \begin{array}{lll}
 P'_{E_1} (F) & = &P_{E_1}(F) + \xi^+(F) \\
 P'_{E_2} (F) &= &P_{E_2}(F) + \xi^-(F) \\
 P'_{E_3} (F) &= &P_{E_3}(F) + \zeta^+(F) \\
 P'_{E_4} (F) &= &P_{E_4}(F)  + \zeta^-(F) 
\end{array}
\right.
\end{equation}
where $\xi^+(F)$ and $\xi^-(F)$  are  multi-branched functions which are smooth in $F^{1/k_1}$ 
and have the same formal expansion, 
where $k_1$ is the natural number such that  $F$ has the same multiplicity at the 
local edges near $p_1$ as the function $x_1^{k_1a_1}y_1^{k_1b_1}$
(i.e. the multiplicity $m_{E_1}$ of $F$ at $E_1$ is equal to $k_1b_1$ and so on), and the same holds true for
$\zeta^+(F)$ and $\zeta^-(F)$. Besides the fact that 
$\xi^\pm (F)$ and $\zeta^\pm (F)$ must be regular in $F^{1/k_1}$, they can be chosen arbitrarily (i.e. we can choose the corresponding
multi-branched functions $\rho^\pm$ and $\theta^\pm$ in order to get the desired functions $\xi$ and $\zeta$).
In other words, we have a {\bf coboundary} of the type
\begin{equation} \label{eqn:Period_Coboundary}
\left\{ 
 \begin{array}{lllll}
 P_{E_1}  & = & \xi^+(F) &= &\hat{\xi}^+(F^{1/k_1}) \\
 P_{E_2}  &= & \xi^-(F) &= &\hat{\xi}^-(F^{1/k_1}) \\
 P_{E_3}  &= & \zeta^+(F)& =& \hat{\zeta}^+(F^{1/k_1}) \\
 P_{E_4}  &= & \zeta^-(F) &=& \hat{\zeta}^-(F^{1/k_1}) 
\end{array}
\right.
\end{equation}
($\hat{\xi}^\pm$ and $\hat{\zeta}^\pm$ are smooth functions; the other components are zero), 
and these are the generators of the space of coboundaries which can be obtained
by changes of coordinates.

\begin{remark}
 A-priori, each component $P_E$ of the cocycle is only regular in $F^{1/m_E}$, why the
coboundaries have a higher level of regularity (each ``elementary'' coboundary is regular in some
$F^{1/k_i}$). Due to this fact, if $P_E \neq 0$ for some edge $E$ we cannot 
in general substract from our cocycle a coboundary so that $P_E$ becomes 0.
\end{remark}

\begin{defn}
 The class of the  period cocycle $(P_E, \ E\ \text{is an edge of}\ N)$ 
in the quotient space of the linear space of all period cocycles 
by the linear space of all coboundaries (generated by the coboundaries given by 
Formula \eqref{eqn:Period_Coboundary}) is called its {\bf cohomology class}.
\end{defn}

Similarly to \cite{DufourMolino-Class2dim1994,San_Focus2003}, since the coboundaries can be multi-branched,
we can use them to kill all the flat terms in the cocycles, i.e. any two cocycles which are the same up to
a flat term are cohomologic. Thus the cohomological class of the cocycle $(P_E, \ E\ \text{is an edge of}\ N)$  depends
only on its asymptotic expansion. This asymptotic expansion is a (multi-dimensional) {\bf Puiseux series} in $F$ (where $F$
is viewed as a local coordinate function on the Reeb graph $\cB$). The cohomology class itself
can be expressed in terms of a Puiseux series with a family of coefficients equal to 0 (those which can be eliminated
by a normalizing coboundary).

The frequency functions $h_i(x_i^{a_i} y_i^{b_i})$ in Formula \eqref{eqn:NFIII_p_i} 
near $p_i$ will also be viewed as functions
of $F$. Similarly to the peiod cocycle, these frequency functions are not regular in $F$, 
but regular in a fractional power of $F$,
and they are determined by the system only up to a flat term, 
so we will also retain only the Puiseux series of these functions in $F$.

\begin{remark}
 In complex analysis there is a problem of multi-valuedness of Puiseux series, but we don't have this problem here
with our systems on real manifolds, because by $F^{1/k}$ of course we mean the unique real $k$-th root of $F$ if $k$ is odd, 
and if $k$ is even (in which case $F$ must also be positive) we usually mean the unique positive real root.
\end{remark}

\begin{thm}
 Let $(X_1,\cF_1)$ and $(X_2,\cF_2)$ be two smooth  integrable systems with two
respective nondegenerate singular level sets $N_1$ and $N_2$ of type III. Then these two systems are semi-locally isomorphic, i.e. 
there is a smooth diffeomorphism from a neighborhood of $N_1$ to a neighborhood of $N_2$ which sends $X_1$
to $X_2$ if and only if they are semi-locally orbitally equivalent and satisfy the following additional condition: there
is a local coordinate function $F_1$ (resp. $F_2$) on the local base space $\cB_{\cU(N_1)}$ (resp. $\cB_{\cU(N_2)}$)
such that the Puiseux series in $F_1$ of the frequency functions and the 
cohomology class of the period cocycle  of $X_1$ coincide with the Puiseux series in $F_2$ of the corresponding 
frequency functions and the the cohomology class of the period cocycle  of $X_2$.
\end{thm}

\begin{proof}
The proof is straightforward, based on the above discussions and the standard gluing and extension methods, 
like in, e.g., \cite{Vey-Isochore1979,DufourMolino-Class2dim1994,DullinSan_Symplectic2007,San_Focus2003}.
\end{proof}

\section{Generic nilpotent singularities (Type IV)}

\subsection{Local normal form}

Assume that $X(p) = 0$ at some point $p$, and that the linear part of $X$ at $p$ is non-zero nilpotent,
i.e. is has Jordan form $X^{(1)} = y\frac{\partial}{\partial x}$. According to a classical result of Takens
\cite{Takens-NF1974}, there exists a formal coordinate system in which $X$ can be formally written as
\begin{equation}X = \big(y + x^2f(x)\big)\frac{\partial}{\partial x} + x^2g(x)\frac{\partial}{\partial y},
 \end{equation}
where $f$ and $g$ are formal series.

Let us assume, moreover, that $dF(p) \neq 0$ for some $F \in \cF$, i.e.
$F$ is a regular function at $p$. Then, according to a theorem of Gong \cite{Gong-NF1995}, 
in the above expression of $X$ one can put $g(x)= 0$, i.e. $X$ can be formally written as 
\begin{equation}X = \big(y + x^2f(x)\big)\frac{\partial}{\partial x},
 \end{equation}
and $F$ is a function of $y$.

We will say that the nilpotent singularity of $(X,\cF)$ at $p$ is \emph{generic}, if $f(0) \neq 0$ in the
above formal normal form. Another equivalent definition of genericity, without using Takens-Gong normal form,
is as follows:
\begin{defn} \label{defn:Generic_Nilpotent}
Let $p$ be a nilpotent singular point of an integrable system $(X, \cF)$. Then we will say that $p$ is
a {\bf generic} nilpotent singular point if the following conditions are satisfied:

i) There exists a local smooth first integral $F$ such that $dF(p)\neq 0$.

ii) There is a local smooth coordinate system $(x,y)$ in which $F$ is a function of $y$, and 
\begin{equation}\frac{\partial^2}{\partial x^2}X := \big[\frac{\partial}{\partial x},[\frac{\partial}{\partial x},X]\big] \neq 0.
 \end{equation}
\end{defn}
Let $p$ be a generic nilpotent singularity of $(X, F)$. Then there is a coordinate system $(x,y)$ in a neighborhood 
$\cU$ of $p$ such that $y$ is a local first integral of $X$, and the linear part of $X$ is in Jordan form, i.e.
\begin{equation}X = \big(y + G(x,y)\big)\frac{\partial}{\partial x},
 \end{equation}
where $G(x,y)$ is a smooth function such that
\begin{equation}G(0,0) = \frac{\partial G}{\partial x}(0,0)= \frac{\partial G}{\partial y}(0,0) =0.
 \end{equation}
Moreover, according to our assumptions,
\begin{equation}\frac{\partial^2 G}{\partial x^2}(0,0) \neq 0.
 \end{equation}
According to the implicit function theorem, for each $x$ near 0, there is a unique $y = \gamma(x)$ such that
\begin{equation}G\big(x,\gamma(x)\big)+\gamma(x)=0.
 \end{equation}
Moreover, the function $x \mapsto \gamma(x)$ is a smooth function such that $\displaystyle \gamma(0) =0, \frac{d\gamma}{dx}(0) =0$ 
but $\displaystyle \frac{d^2 \gamma}{dx^2}(0) \neq 0$, i.e. $\gamma$ has Morse singularity at 0.

The points $(x,y)=\big(x,\gamma(x)\big)$ in the neighborhood $\cU$ of $p$ are precisely those points at which $X$
vanishes. By Morse theorem, using a smooth change of coordinates, we may assume that the curve
\begin{equation}\cS = \{X=0 \}\cap  \cU = \{(x,y)\in \cU \ |\ y = \gamma(x)\}
 \end{equation}
is the standard parabolic curve $\{ y = x^2 \}$, i.e. $\gamma(x) = x^2$. So we have
\begin{equation}X = \widetilde G(x,y)\frac{\partial}{\partial x},
 \end{equation}
where $\widetilde G = y + G(x,y)$ vanishes on the curve $\{ y = x^2 \}$.
It means that $\widetilde G$ is divisible by the function $y -x^2$, i.e. we can write
\begin{equation}X = g(x,y).(y-x^2)\frac{\partial}{\partial x},
 \end{equation}
where $g(x,y)$ is a smooth function such that $g(0,0) \neq 0$.

Notice that $X$ has eigenvalue equal to 
\begin{equation}-2xg(x,x^2)
 \end{equation}
at each point $(x,x^2) \in \cS$. The function $x[g(x,x^2) -g(-x,x^2)]$
is an even function in $x$, and so can be considered a smooth function in $x^2$ which vanishes at 0.
Denote by $g_1$ a smooth function of one variable such that
\begin{equation}xg_1(x^2) = [g(x,x^2) -g(-x,x^2)]/2.
 \end{equation}
Similarly, there is a smooth function $g_0$ such that
\begin{equation}g_0(x^2) = [g(x,x^2) +g(-x,x^2)]/2.
 \end{equation}
Put 
\begin{equation}X_1 = \big(g_0(y)+xg_1(y)\big)(y-x^2)\frac{\partial}{\partial x}.
 \end{equation}
Then $(X_1,y)$ is also a local smooth integrable system with a generic nilpotent singularity at $p$,
and moreover the set $\{ X_1=0 \}$ locally coincides with the set $\{ X=0 \}$,
and $X_1$ has the same eigenvalue as $X$ at every point of this set.

\begin{prop}\label{prop:nilpotent1}
With the above notations, the vector fields
\begin{equation}X = g(x,y).(y-x^2)\frac{\partial}{\partial x}
 \end{equation}
and 
\begin{equation}X_1 = \big(g_0(y)+xg_1(y)\big)(y-x^2)\frac{\partial}{\partial x}
\end{equation}
are locally smoothly isomorphic. More precisely, there is a smooth local diffeomorphism
$\Phi : (\cU,0) \to (\cU,0)$, which preserves the coordinate $y$ and such that $\Phi_*X = X_1$.
\end{prop}
\begin{proof}
We will use Moser's path method \cite{Moser-VolumeElement}. 
See Appendix A1 of \cite{DufourZung-Poisson2005} for an introduction to this
method. Take the following path of vector fields:
\begin{equation}
X_t = tX + (1-t)X_1 = \rho_t(x,y).(y-x^2)\frac{\partial}{\partial x},
\end{equation}
where $\rho_t = tg+(1-t)(g_0+xg_1)$.

The main point is to show the existence of a time-dependent vector field
\begin{equation}Z_t = \varphi_t(x,y)(y-x^2)\frac{\partial}{\partial x}
 \end{equation}
such that 
\begin{equation}-\cL_{Z_t}X_t = [X_t,Z_t] = \frac{d}{dt}X_t = X_1 - X.
\end{equation}

If $Z_t$ exists, then its time-1 flow will be the required local diffeomorphism which moves $X$ to $X_1$.
Thus we have to solve the equation
\begin{equation}\label{eq:1flow}
\big[\rho_t(x,y).(y-x^2)\frac{\partial}{\partial x}, \varphi_t(x,y)(y-x^2)\frac{\partial}{\partial x}\big]=X_1 - X.
\end{equation}

Notice that $X_1 - X$ vanishes up to the second order on the curve
$\cS = \{(x,y) \in \cU \ |\ y = x^2\}$, i.e. we can write
\begin{equation}X_1 - X = h(x,y).(y-x^2)\frac{\partial}{\partial x}.
\end{equation}

Then Equation \eqref{eq:1flow} is equivalent to 
\begin{equation}\rho_t \varphi_t' - \varphi_t \rho_t'= h
 \end{equation}
where the apostrophe means derivation by $x$. This last equation is equivalent to
\begin{equation}\big(\frac{\varphi_t}{\rho_t}\big)' = \frac{h}{\rho_t^2},
 \end{equation}
which admits a smooth solution
\begin{equation}\varphi_t(x,y)= \rho_t(x,y)\int_0^x \frac{h(s,y)}{\rho_t^2(s,y)} ds.
 \end{equation}
The proposition is proved.
\end{proof}
Proposition \ref{prop:nilpotent1} means that any generic nilpotent singularity of a smooth integrable
system $(X,F)$ has the following smooth normal form:
\begin{equation}X = \big(g_0(y)+xg_1(y)\big)(y-x^2)\frac{\partial}{\partial x},
 \end{equation}
where $g_0(0)\neq 0$ (one can fix $g_0(0)=1$ if one wishes).

Notice that the eigenvalues of the singular point 
$(x,y) = (x,x^2) = (\pm \sqrt y,y)$ of $X$ is 
\begin{equation}\mp 2\sqrt y\big(g_0(y)\pm \sqrt y g_1(y)\big) = -2x\big(g_0(x^2)+xg_1(x^2)\big).
 \end{equation}
Of course, these eigenvalues depend only on $X$ and on the parametrization of the local singular curve 
$\cS = \{X =0\}$ by the coordinate $x$ (such that locally $\cS = \{(x,x^2) \in \bbR^2 \ |\ x \in \bbR\}$).

If $\{(x_1,y_1 =x_1^2) \ |\ x_1 \in \bbR\}$ is another parametrization of $\cS$ in another smooth coordinate
system $(x_1,y_1)$, then there is an odd function $\psi$ such that $\psi(0) =0, \psi'(0) \neq 0, 
\psi(-x) =-\psi(x)$ and 
\begin{equation}x_1 = \psi(x), y_1 = \psi(\sqrt y)^2 = \psi(x)^2 \text{ on } \cS.
 \end{equation}
Thus the function 
\begin{equation}E: x \mapsto -2x\big(g_0(x^2)+xg_1(x^2)\big),
 \end{equation}
considered up to composition by reversible odd functions (i.e. $E$ is equivalent to $E\circ \psi$ for any $\psi$ odd reversible) 
is a local invariant of $X$ at $p$. We will call the equivalence $E \sim E\circ \psi$ the {\bf odd left equivalence} class of $E$.

The local function
\begin{equation} \label{eqn:E_function}
E(x)= -2x\big(g_0(x^2)+xg_1(x^2)\big)
 \end{equation}
(considered up to odd left equivalence) will be called the {\bf eigenvalue function} or also the {\bf frequency function} 
of $X$ at $p$. According to the above discussion
and Proposition \ref{prop:nilpotent1}, this eigenvalue function is the full invariant of $(X,F)$ 
at the generic nilpotent singular point $p$. In other words, we have proved the following theorem:

\begin{thm} \label{thm:nilpotent-NF}
1) Let $p$ be a generic nilpotent singularity of a smooth integrable system $(X,\cF)$. Then there is a local smooth coordinate
system $(x,y)$ in a neighborhood $\cU$ of $p$ in which $X$ has the following normal form:
\begin{equation}X = \big(g_0(y)+xg_1(y)\big)(y-x^2)\frac{\partial}{\partial x},
\end{equation}
where $g_0, g_1$ are two smooth functions and $g(0) \neq 0$.

2) If $\displaystyle \widetilde X = \big(\widetilde g_0(y_1)+x_1\widetilde g_1(y_1)\big)(y_1-x_1^2)\frac{\partial}{\partial x_1}$
is another generic nilpotent integrable smooth vector field in another smooth local normal form, then $X$ and $\widetilde X$ 
are locally smoothly isomorphic if and only if their respective eigenvalue functions $E(x)= -2x\big(g_0(x^2)+xg_1(x^2)\big)$ and 
$\widetilde E(x_1)= -2x_1\big(\widetilde g_0(x_1^2)+x_1\widetilde g_1(x_1^2)\big)$
are locally odd left equivalent, i.e. there is a local smooth odd reversible function $\psi$ such that $\widetilde E\big(\psi(x)\big) = E(x)$.
\end{thm}

The eigenvalue function $E(x)$ in Formula \eqref{eqn:E_function} and Theorem \ref{thm:nilpotent-NF}
is a smooth function, but it is not very convenient for the semi-local and global study, because the variable $x$
is not a first integral of the system. So instead of $E(x)$ we will consider the function
\begin{equation} \label{eqn:DoubleValuedEigenvalue}
 F(y) = - 2 \sqrt{y} (g_0(y) + \sqrt{y} g_1 (y)).
\end{equation}
This function $F(y)$ is double-valued and not smooth in $y$ (it is smooth only in $\sqrt{y}$), but since $y$ is a first integral
in the local normal form, we can project $F$ to a double-valued function on the local base space. $F$ will be called
the {\bf double-valued eigenvalue function}.  Another advantage of $F$ is that, instead of odd left equivalence,
we now have usual left equivalence, i.e. for any local smooth diffeomorphism $\phi: (\bbR_+,0) \to (\bbR_+,0),$
$F$ is equivalent to $F \circ \phi.$ The second part of Theorem \ref{thm:nilpotent-NF} can be now restated as follows:

\begin{thm}
Two singularities of Type IV are smoothly locally isomorphic if and only if their corresponding double-valued eigenvalue
functions are left equivalent. 
\end{thm}

\subsection{Complexification and regularized monodromy}

Consider a generic nilpotent smooth integrable vector field $X$ in normal form:
\begin{equation}X = \big(g_0(y)+xg_1(y)\big)(y-x^2)\frac{\partial}{\partial x}.
 \end{equation}
\begin{figure}[htb] 
\begin{center}
\includegraphics[width=70mm]{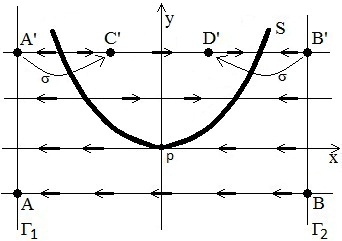}
\caption{Type IV singular point}
\label{fig:nil-singularities}
\end{center}
\end{figure}

Fix two local curves $\Gamma_1$ and $\Gamma_2$ which are transversal to the lines $\{y = constant\}$, and which 
lie on the two different sides of the nilpotent singular point $p$ in the above coordinate system $(x,y)$. For example, one can take 
$\Gamma_1 = \{x = -c\}$ and $\Gamma_2 = \{x = c\}$ for some small positive constant $c$.

For each point $A \in \Gamma_1$ there is a unique point $B \in \Gamma_2$ such that $y(A) = y(B)$, i.e. $A$ and $B$ lie
on the same local smooth invariant curve of $X$.

If $y(A) < 0$ then locally the vector field $X$ is non-singular on the invariant curve $\{ y = y(A)\}$, and so there is 
a unique time $T$ (if $g_0(0)> 0$ then $T<0$, and if $g_0(0)< 0$ then $T>0$) such that the time-$T$ flow $\varphi_X^T$ 
of $X$ moves $A$ to $B$
\begin{equation}\varphi_X^T (A) = B.
\end{equation}
Of course, $T$ depends on the choice of  $\Gamma_1, \Gamma_2$ and the value of $y = y(A)$, so we can write it as a function
of $y$:
\begin{equation}T = T_{\Gamma_1, \Gamma_2}(y).
\end{equation}
It is clear that 
\begin{equation}\lim_{y \to 0^-} T_{\Gamma_1, \Gamma_2}(y) = \infty
\end{equation}
because $X(p) = 0$, thus the function $T_{\Gamma_1, \Gamma_2}$ is singular at $y = 0$.
We want to regularize this function, i.e. write it as
\begin{equation}T_{\Gamma_1, \Gamma_2}(y) = \cS(y) + \widehat T_{\Gamma_1, \Gamma_2}(y)
\end{equation}
where $\widehat T_{\Gamma_1, \Gamma_2}(y)$ is a smooth function, and $\cS$ is a singular function which
does not depend on the choice of $\Gamma_1$ and $\Gamma_2$.

If $y = y(A)> 0$ then the invariant curve $\{ y = y(A)\}$ contains two singular points $(\pm \sqrt y, y)$, and we can't even go from $B$
to $A$ by the flow of $X$. But, as was shown in \cite{Zung-Nondegenerate2012} and in Subsection
\ref{subsection:semilocal_II}, 
we can go from $B$ to $A$ using the flow $\varphi_X$ by ``jumping over the walls'' as follow:
\begin{figure}[htb] 
\begin{center}
\includegraphics[width=60mm]{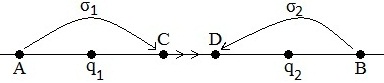}
\caption{Jumping over the walls.}
\label{fig:jumping}
\end{center}
\end{figure}

Denote by $q_1$ and $q_2$ the two nondegenerate hyperbolic singular points with eigenvalue 0 on the invariant curve $\{ y = y(A) = y(B)\}$.
Denote by $\sigma_1$ and $\sigma_2$ the two reflection maps associated to  $q_1$ and $q_2$ respectively. 
Put $C = \sigma_1(A), D = \sigma_2(B)$,
and define the regularized time function $\widehat T_{\Gamma_1, \Gamma_2}(y)$  when $h>0$ by the formula
\begin{equation} \label{eqn:Time_IVa}
\varphi_X^{\widehat T_{\Gamma_1, \Gamma_2}(y)}(C) = D.
\end{equation}
We will show that the singular function $\cS(y)$ can be chosen in such a way that the regularized 
time function $\widehat T_{\Gamma_1, \Gamma_2}(y)$
for $y>0$ agrees with  $\widehat T_{\Gamma_1, \Gamma_2}(y)$ for $y<0$ to become together a smooth function of $y$ at $y =0$.

For simplicity, let us first look at the analytic case, i.e. the case when the functions $g_0$ and $g_1$ in the normal form
\begin{equation}X = \big(g_0(y)+xg_1(y)\big)(y-x^2)\frac{\partial}{\partial x}
\end{equation}
are real analytic functions. In this analytic case, we can use the complexification method to study the regularized time function.
 
 By complexification, the analytic vector field in a neighborhood of 0 in $\bbC^2$. On each local invariant complex line 
 $\cC_y = \{y = \text{ constant }\}$, $X$ vanishes at exactly two points $q_{1,2} = \pm \sqrt y$. In the complex line $\cC_y$, we
 can go from $A$ to $B$ by a path $\gamma$ which avoids $q_1$ and $q_2$ as shown in the Figure \ref{fig:Goingpath}.
\begin{figure}[htb] 
\begin{center}
\includegraphics[width=60mm]{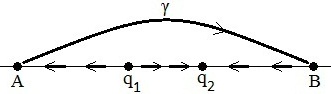}
\caption{Going from $A$ to $B$ using complex path $\gamma$.}
\label{fig:Goingpath}
\end{center}
\end{figure}

Then we can put $T_{\Gamma_1, \Gamma_2}^\bbC(y)$ equal to the complex time to go by $\varphi_X$ along $\gamma$ from $A$ to $B$.
The formula for $T_{\Gamma_1, \Gamma_2}^\bbC(y)$ is:
\begin{equation}T_{\Gamma_1, \Gamma_2}^\bbC(y) = \int_\gamma \frac{dx}{\big(g_0(y)+xg_1(y)\big)(y-x^2)}.
 \end{equation}
Of course, the above formula depends only on the homotopy class of $\gamma$. When $y \in \bbR, y> 0$, we can also imagine $\gamma$ 
as in Figure \ref{fig:Goinghalf-circle}:
\begin{figure}[htb] 
\begin{center}
\includegraphics[width=60mm]{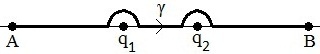}
\caption{Going half-circle around $q_1$ and $q_2$.}
\label{fig:Goinghalf-circle}
\end{center}
\end{figure}

The above path $\gamma$ consist of 3 pieces of real time (on the real path from $A$ to $B$) and two half-circles of pure imaginary time
(going half-circle around $q_1$ and $q_2$). The total real time is actually equal to $\widehat T_{\Gamma_1, \Gamma_2}(y)$,
while the imaginary time for going half-circle around each $q_i$ is equal to $-i\pi$ divided by the eigenvalue of $X$ at $q_i$. Thus we have
\begin{equation}
\widehat T_{\Gamma_1, \Gamma_2}(y) = \Re\big(T_{\Gamma_1, \Gamma_2}^\bbC(y)\big)
\end{equation}
where $\Re$ denotes the real part.

Similarly, when $y \in \bbR_-$, we can imagine the path $\gamma$ as in Figure \ref{fig:Goingaroundq1}.
\begin{figure}[htb] 
\begin{center}
\includegraphics[width=60mm]{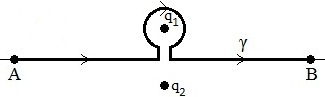}
\caption{Going around $q_1$.}
\label{fig:Goingaroundq1}
\end{center}
\end{figure}

The time for going from $A$ to $B$ along $\gamma$ is equal to the time going from $A$ to $B$ along the real path 
plus the time for going around the singular point $q_1 = (i\sqrt{-y},y)$ in the negative direction of the complex plane.
The time for going around $q_1$ in the positive direction is equal to $2\pi\sqrt{-1}$ divided by the eigenvalue of $X$
at $q_1$, so it is a complex number whose real part is (for $y < 0$):
\begin{equation} \label{eqn:SingularPart_IV}
 S(y) = \frac{\pi \sqrt{-y} g_0(y)}{-y(g_0^2(y) - y g_1^2(y))}.
\end{equation}
Thus we can put, for $y \leq 0$:
\begin{equation} \label{eqn:Time_IVb}
 \widehat{T}_{\Gamma_1,\Gamma_2}(y) = {T}_{\Gamma_1,\Gamma_2}(y) - \frac{\pi \sqrt{-y} g_0(y)}{-y(g_0^2(y) - y g_1^2(y))}.
\end{equation}
Formula \eqref{eqn:Time_IVb} makes sense also in the smooth non-analytic case, and together with Formula \eqref{eqn:Time_IVa}
gives us a local smooth function in $y$. This local smooth function given by Formula  \eqref{eqn:Time_IVa} for $y > 0$
and by Formula \eqref{eqn:Time_IVb} for $y \leq 0$ is called the {\bf regularized time function} for going
from $\Gamma_1$ to $\Gamma_2$ by the flow of $X$.

The time function $ \widehat{T}_{\Gamma_1,\Gamma_2}$ depends on $\Gamma_1$ and $\Gamma_2$. In order to make it into something
independent of $\Gamma_1$ and $\Gamma_2$, we go back by the flow of $X$ from $ \Gamma_2$ to $\Gamma_1$, not in the same way, but
in the other way ``around the globe'' to make a loop (a closed level set). 

Assume, for simplicity,  that the level set $\cN$ of the Type IV singular point $p$  does not
contain any singular point of Type III, and it does not contain any other Type IV singular point either, 
and moreover the neighborhood $N$  in the surface $\Sigma$ is orientable. Then, 
similarly to the case of Type II level sets, it is easy to see that  $\cN$  is a regular level set of the associated fibration. Denote by
\begin{equation}
R_{\Gamma_2,\Gamma_1} (y)
\end{equation}
the time function for returning from $\Gamma_2$ to $\Gamma_1$ ``by going around the globe'', i.e. not by the previous path $\gamma$, but by the
complementary path in each level set. Of course, if that complementary path crosses some Type II singular points, 
then we will jump over them as we did before, and the time function is still well defined. The function
\begin{equation}
 \widehat{M}(y) = \widehat{T}_{\Gamma_1,\Gamma_2}(y) + R_{\Gamma_2,\Gamma_1} (y)
\end{equation}
is called the {\bf regularized monodromy function} near the level set $\cN$.

Observe that, when $y > 0$ then $\widehat{M}(y) = M(y)$ is the usual monodromy of a regular level set with Type II singular points, 
so the function $\widehat{M}(y)$ for 
$y \geq 0$ is, up to left equivalence, a semi-local invariant of the system. On the other hand, similarly to the case of Type III level sets,
$\widehat{M}(y)$ for $y < 0$ considered up to left equivalence is NOT an invariant, only its Taylor series at $y=0$ is. But this Taylor
series is also determined by $\widehat{M}(y)$ for $y > 0,$ so we can actually forget about $\widehat{M}(y)$ for $y < 0$ in the semi-local
classification of Type IV level sets. The function $\widehat{M}(y)$ for $y \geq 0$ will be called the {\bf truncated monodromy function}
(we truncated the part with $y < 0$). Summarizing, we have:

\begin{thm}
Let $p$ be a Type 4 singular point of a smooth integrable system $(X,\cF)$ on a compact surface $\Sigma$, 
such that the level set $\cN$ of $p$ does ot contain any other singular point of Type III or Type IV, and
the surface $\Sigma$ is orientable near $\cN$. Then $\cN$ is a regular level set of the associated fibration. Moreover, 
the local eigenvalue functions in a neighborhood $\cU(N)$ of $N$ and the truncated monodromy function in 
$\cU(N)$, considered as functions on the local base space andup to simultaneous local left equivalence, 
classify $(X,\cF)$  up to semi-local smooth isomorphisms.
\end{thm}

The case when $\cU(N)$ is no-orientable can be classified similarly: in this case it is the Taylor series of the regularized
monodromy function which is the continuous semi-local invariant (in additional to the local invariants).  The case when $\cN$
contains other singular points of Type IV and Type III is more complicated: in this case, instead of the regularized monodromy,
we have to talk about regularized period cocycles, similarly to the case of level sets of Type III. We don't want to go
into the details here.

\section{Global classification}

In order to obtain a global classification of integrable systems $(X,\cF),$ we just need to collect all the local and semi-local invariants. 
(There is no specifically global invariant, due to the fact that the base space of the associated singular fibration  has only 1 dimension 
and does not carry any additional structure on it besides the smooth structure). We can formulate the following classification theorems, 
whose proof is simply a combination of the results of the previous sections:

\begin{thm}
Two smooth integrable systems $(X_1,\cF_2)$ and $(X_2,\cF_2)$  on closed surfaces $\Sigma_1$ and $\Sigma_2$ respectively
are smoothly orbitally equivalent if any only if there is a homeomorphism $\Phi: \Sigma_1 \to \Sigma_2$ which is a bijection
when restricted to the sets of singular points of Type $k$ of $X_1$ and $X_2$ for each $k=$ I, II, III, IV, and such that for each
singular point $p \in \Sigma_1$ of Type III the ratio of the two eigenvalues of $X_1$ at $p_1$ is equal to the ratio of the
two eiganvalues of $X_2$ at $\Phi(p_1)$.
\end{thm}

\begin{thm}
 Two smooth integrable systems $(X_1,\cF_1)$ and $(X_2,\cF_2)$  on closed surfaces $\Sigma_1$ and $\Sigma_2$ respectively
are smoothly isomorphic if any only if  there is a smooth diffeomorphism 
$\Phi: \Sigma_1 \to \Sigma_2$ which is a smooth orbital equivalence of the two systems, such that
the quotient map $\widehat{\Phi}: \cB_1 \to \cB_2$ from the quotient space $\cB_1$ to the quotient space
$\cB_2$ is a simultaneous left equivalence between the following objects of $(X_1,\cF_1)$ (considered as functions on $\cB_1$)
and the corresponding objects of $(X_2,\cF_2)$: the period functions and the monodromy functions (for regular level sets which
may contain Type II singular points, and also near Type I singular points), the eigenvalue functions
(for Type II singularities), the truncated monodromy functions (for simple Type IV level sets), the Taylor series of the regularized
monodromy function (for twisted Type IV level sets), the Puiseux series of the frequency functions 
and the cohomology classes of the period cocycles (for Type III level sets and also for 
mixed Type III - Type IV level sets).
\end{thm}

\section{Hamiltonianization}

We say that a vector field $X$ on a manifold $M$ is {\bf Hamiltonianizable}, or that it admits a {\bf hamiltonianization}, if
there exists a function $H$ and a Poisson structure $\Pi$ on $M$ such that $X = X_H := dH \lrcorner \Pi.$ We will distinguish
the symplectic case (when $\Pi$ is nondegenerate) from the degenerate case (when $\Pi$ vanishes at some points).

\begin{thm} \label{thm:Hamilton_Symplectic}
Let $(X,\cF)$ be a weakly nondegenerate smooth integrable system on a compact surface $\Sigma$. Then $X$ is
Hamiltonianizable by a symplectic structure if and only if the following conditions are satisfied: \\
i) Every singular point of Type III is traceless, i.e. the sum of the two eigenvalues of $X$ at that point is 0, \\
ii) There is a global smooth coordinate function on the base space $\cB,$ \\
iii) The surface $\Sigma$ is orientable, \\
iv) $(X,\cF)$ does not contain singular points of Type II and Type IV.
\end{thm}
\begin{proof}
It is clear that the above conditions are necessary for $X$ to be Hamiltonianized, because symplectic manifolds are orientable,
symplectic vector fields have zero trace at singular points, and the Hamiltonian function of the system will project to a global
coordinate function on the base space. Let us show that these conditions are also sufficent.

Condition i) implies that  $X$ can be written as 
$X = g(xy)\big(x\frac{\partial}{\partial x}-y\frac{\partial}{\partial y}\big)$ near each singular point of Type III. 
According to Proposition \ref{prop:TypeIII_1}, 
near every hyperbolic level set $\cN$ (i.e. a level set which contains a singular point of Type III) there is a first integral $F_\cN$
without multiplicity at $\cN$, i.e. the order of vanishing of $F_\cN$ at every component of $\cN$ is 1, or in the other words, 
all singular points of $F$ in $\cU(\cN)$ are nondegenerate. According to local geometric linearization theorem , near an elliptic singular
point there is also a nondegenerate first integral of type $F = x^2 + y^2$ (see Section 2). Conditions ii) and iii) then implies 
that there is a global first integral $H$ which is a Morse function on $\Sigma.$

The symplectic form can be chosen locally in such a way that $X=X_H$, i.e. $X$ is the Hamiltonian vector field of the above 
Morse first integral $H$: 

a) In a neighborhood $\cU(p_i)$ of an elliptic singular point $p_i$ where 
$X = g(x^2+y^2)\big(x\frac{\partial}{\partial x}-y\frac{\partial}{\partial y}\big)$ and $H = h(x^2+y^2)$, put
\begin{equation}
w_i = -\frac{g(x^2+y^2)}{2h'(x^2+y^2)}dx\wedge dy.
\end{equation}
where $h'$ is the derived function of $h$.

b) In a neighborhood $\cU(q_j)$ of a hyperbolic singular point $q_j$ where 
$X = g(xy)\big(x\frac{\partial}{\partial x}-y\frac{\partial}{\partial y}\big)$ and $H = h(xy)$, put
\begin{equation}
w_j = -\frac{g(xy)}{h'(xy)}dx\wedge dy.
\end{equation}
c) In a neighborhood $\cU(r_k)$ of a regular point $r_k$ where 
$X = f(y)\frac{\partial}{\partial x}$ and $H = h(y)$ put
\begin{equation}
w_k = \frac{f(y)}{h'(y)}dx\wedge dy.
\end{equation}
We can choose the neighborhoods $\cU(p_i), \cU(q_j), \cU(r_k)$ so that 
they form a finite open covering of $\Sigma$, and the local canonical coordinate systems in them so that
the above symplectic form $w_i,w_j,w_k$ give the same orientation of $\Sigma$. Let
\begin{equation}
1 = \sum_i\psi_i + \sum_j\psi_j + \sum_k\psi_k
\end{equation}
be a partition of unity on $\Sigma$ such that $\psi_i \geq 0$ everywhere and $\psi_i = 0$ outside of $\cU(p_i)$,
and similarly for $\psi_j$ and $\psi_k$. Put 
\begin{equation}
w = \sum_i\psi_i w_i + \sum_j \psi_j w_j + \sum_k \psi_k w_k.
\end{equation}
Then $w$ is a global symplectic form on $\Sigma$, and we have 
\begin{equation}
X = X_H
\end{equation}
with respect to $w$ on $\Sigma$.
\end{proof}

\begin{thm} \label{thm:Hamilton_Poisson}
Let $(x,\cF)$ be a weakly nondegenerate smooth integrable system on a compact surface $\Sigma$. Then $X$ is
Hamiltonianizable by a Poisson structure if and only if the following conditions are satisfied:\\
i) Every singular point of Type III is traceless, i.e. the sum of the two eigenvalues of $X$ at that point is 0, \\
ii) There is a global smooth coordinate function on the base space $\cB,$ \\
iii) $\Sigma$ is orientable in the neighborhood of every level set.
\end{thm}

\begin{remark}
The first two conditions in Theorem \ref{thm:Hamilton_Poisson} are the same as in Theorem \ref{thm:Hamilton_Symplectic},
but the last condition in  Theorem \ref{thm:Hamilton_Poisson} is much weaker than the last two conditions in
Theorem  \ref{thm:Hamilton_Symplectic}: in the Poisson case we allow Type II and Type IV singular points (the set of such
points is a disjoint union of closed simple curves on $\Sigma$), and $\Sigma$ is not required to be orientable globally:
it can be non-orientable in the neighborhood of a closed curve of singular points of Type II and Type IV. See
Figure \ref{fig:Poisson_nonorientable} for an example.
\end{remark}

\begin{figure}[htb] 
\begin{center}
\includegraphics[width=80mm]{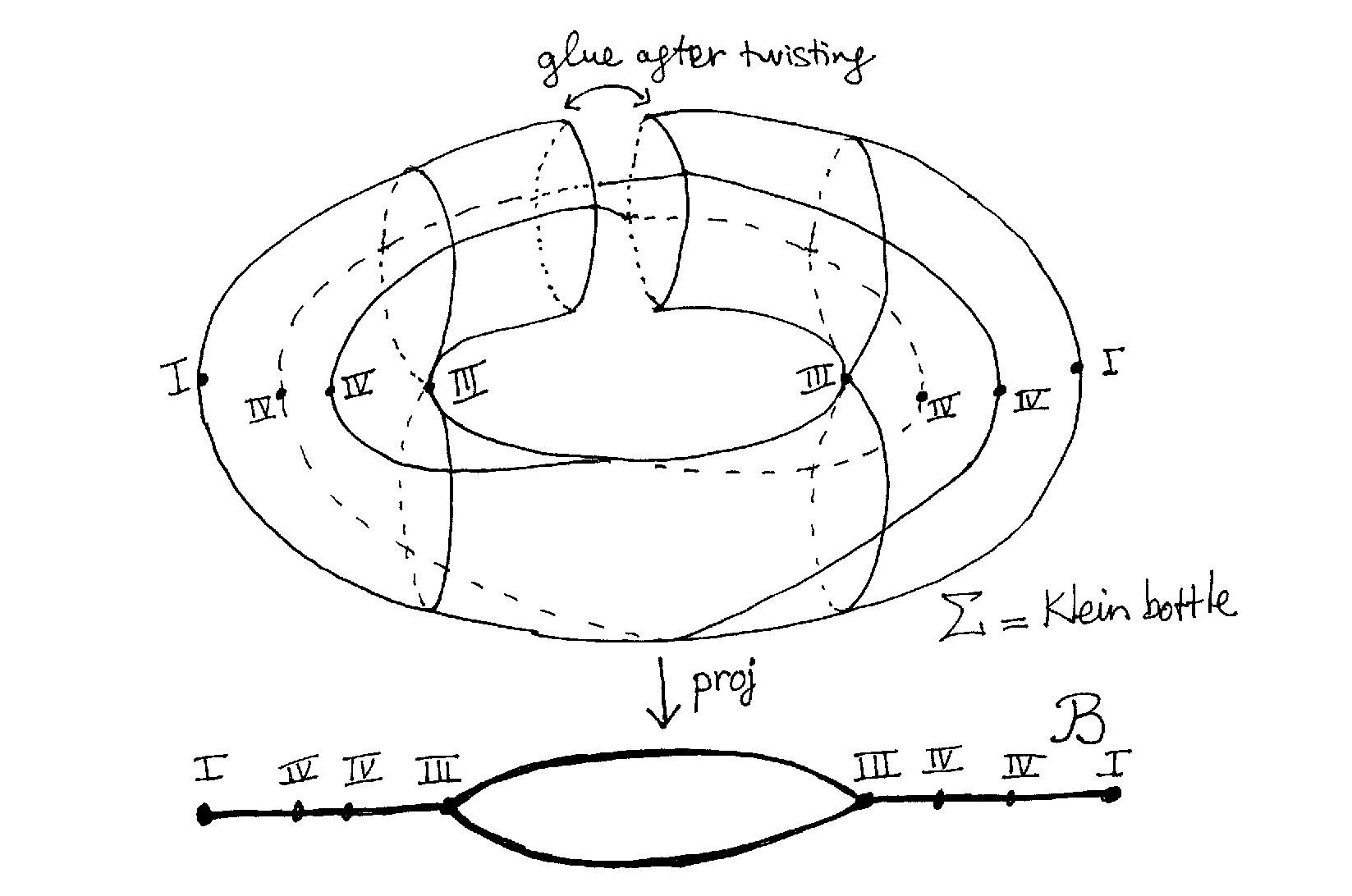}
\caption{A Hamiltonian system on an non-orientable Poisson surface}
\label{fig:Poisson_nonorientable}
\end{center}
\end{figure}

\begin{proof}
The proof is similar to the proof of Theorem \ref{thm:Hamilton_Symplectic}. We still have a global Morse first integral.
We can still construct Poisson structures locally, and then sum them up together by a partition of unity on $\Sigma$.
Near a Type II singular point where 
\begin{equation}
X = \gamma(y) x \frac{\partial}{\partial x} \ \text{and} \ H = h(y)
\end{equation}
we can choose the Poisson structure
\begin{equation}
\Pi = \frac{\gamma(y)}{h'(y)} \frac{\partial}{\partial x} \wedge \frac{\partial}{\partial y}.
\end{equation}
Near a Type IV singular point where
\begin{equation}
X = (y + G(x,y)) \frac{\partial}{\partial x} \ \text{and} \ H = h(y)
\end{equation}
we can choose the Poisson structure
\begin{equation}
\Pi = \frac{y + G(x,y)}{h'(y)} \frac{\partial}{\partial x} \wedge \frac{\partial}{\partial y}.
\end{equation}
The rest of the proof is straightforward. 
\end{proof}

\begin{remark}
 The Poisson structure in Theorem \ref{thm:Hamilton_Poisson} vanishes on the set of singular points of Type II and Type IV.
This set is a disjoint union of regular simple closed curves, and the Poisson structure is locally isomorphic to 
$x\frac{\partial}{\partial x} \wedge \frac{\partial}{\partial y}$ near every of these singular points. 
Such Poisson structures are stable under perturbations 
(see, e.g., \cite{DufourZung-Poisson2005,Radko-Poisson2002}), and generic Hamiltonian systems on them will admit singularities
of Type II and Type IV. This is one more good reason to include 
nilpotent Type IV singularities in our definition of  weaky-nondegenerate integrable systems on surfaces.
\end{remark}

\end{document}